\documentclass[12pt]{amsart}
\usepackage{amsmath, amssymb, amsthm, bm, mathtools, enumitem, caption}
\usepackage{hyperref}

\usepackage[noabbrev,capitalize]{cleveref}
\usepackage{adjustbox}
\crefname{equation}{}{}
\usepackage{graphicx, tikz}
\usepackage[textheight=9in, textwidth=7in]{geometry}

\usepackage{microtype}

\newtheorem{theorem}{Theorem}[section]
\newtheorem{lemma}[theorem]{Lemma}

\newtheorem{proposition}[theorem]{Proposition}

\newtheorem*{conjecture*}{Conjecture}

\theoremstyle{definition}

\theoremstyle{remark}
\newtheorem*{remark}{Remark}

\newtheorem*{examples}{Examples}

\numberwithin{equation}{section}

\DeclarePairedDelimiter\abs{\lvert}{\rvert}

\newcommand{\R}{\mathbb R}
\newcommand{\N}{\mathbb N}

\newcommand{\re}{{\text {\rm Re}}}

\newcommand{\ad}{a_{\Delta}}
 \def\H{\mathbb{H}}

\newcommand{\FF}{\mathcal{F}}

\newcommand{\cA}{\mathcal A}

\newcommand{\cP}{\mathcal P}

\newcommand{\C}{\mathbb C}
\newcommand{\SL}{\mathrm{SL}}

\newcommand{\Z}{\mathbb Z}
\newcommand{\ord}{\mathrm{ord}}
\renewcommand{\abs}[1]{\left\vert #1 \right \vert}

\begin{document}

\title[Hecke polynomials for the mock modular form arising from the Delta-function]{Hecke polynomials for the mock modular form arising from the Delta-function}

\dedicatory{In celebration of Masanobu Kaneko's (60+4)th birthday }
\thanks{2020 {\it{Mathematics Subject Classification.}} 11F30, 11F25}
\keywords{Hecke operators, Ramanujan's Delta-function, mock modular forms}

\author{Kevin Gomez \and Ken Ono}
\address{Dept. of Mathematics, University of Virginia, Charlottesville, VA 22904}
\email{vhe4ht@virginia.edu}
\email{ken.ono691@virginia.edu}

\begin{abstract}
We consider a mock modular form $M_{\Delta}(\tau)$ that arises naturally from Ramanujan's Delta-function. It is a weight $-10$ harmonic Maass form whose nonholomorphic part is the ``period integral function'' of $\Delta(\tau)$. The Hecke operator $T_{-10}(m)$ acts on this mock modular form in terms of Ramanujan's $\tau(m)$ and a monic degree $m$ polynomial $F_m(x),$ evaluated at $x=j(\tau).$  
In analogy with results by Asai, Kaneko, and Ninomiya \cite{AKN} on the zeros of Hecke polynomials for the $j$-function, we prove that the zeros of each $F_m(x)$, including $x=0$ and $x=1728,$ are distinct and lie in $[0, 1728]$. Additionally, as $m \to
+\infty,$ these zeros become equidistributed in $[0, 1728].$

\end{abstract}
\maketitle
\noindent
\section{Introduction and Statement of Results}

For positive even integers $k$, the Eisenstein series $E_{k}(\tau)$ (note. $q:=e^{2\pi i \tau}$) has Fourier series
$$
E_{k}(\tau):= 1 -\frac{k}{B_{k}}\sum_{n=1}^{\infty} \sigma_{k-1}(n)q^n,
$$
where $B_{k}$ is the $k^{{\text {\rm th}}}$ Bernoulli number and
$\sigma_{\nu}(n):=\sum_{d\mid n}d^{\nu}.$ If $k\geq 4$ is even, then $E_{k}(\tau)$ is in $M_k$, the complex vector space of  weight $k$ holomorphic modular form on $\SL_2(\Z)$.
 In the 1970s, F. K. C. Rankin and H. P. F. Swinnerton-Dyer \cite{RSD} showed that the zeros of these modular forms lie on the lower boundary of the standard fundamental domain $\mathcal{F}$ for $\SL_2(\Z),$ the arc of the unit circle 
 \begin{equation}\label{A}
\mathcal{A}:= \left \{ \tau \in \H \ : \ |\tau|=1\ \ {\text {\rm with}}\ \ -\frac{1}{2}\leq \re(\tau)\leq 0 \right\}.
\end{equation}
A straightforward procedure reformulates this result in terms of ``divisor polynomials'' of modular forms, with zeros in the interval $[0, 1728].$

To clarify, we first note that the elliptic points  $\tau\in \{ i, \omega:=e^{2\pi i/3}\}$ are often trivial zeros of modular forms.
Indeed, if $f(\tau)\in M_k,$ then the valence formula (for example, see p. 8 of \cite{OnoCBMS}) implies
$$
\ord_{i}(f)\geq \begin{cases}
 1 \ \ \ \ \ &{\text {\rm if}}\ k\equiv 2\pmod 4,\\
0 \ \ \ \ \ &{\text {\rm if}}\ k\equiv 0\pmod 4,\\
\end{cases}
$$
and
$$
\ord_{\omega}(f)\geq \begin{cases}
 2 \ \ \ \ \ &{\text {\rm if}}\ k\equiv 2\pmod 6,\\
 1 \ \ \ \ \ &{\text {\rm if}}\ k\equiv 4\pmod 6,\\
0 \ \ \ \ \ &{\text {\rm if}}\ k\equiv 0\pmod 6.
\end{cases}
$$
The trivial zeros at $\{i, \omega\}$ are precisely captured by the low weight Eisenstein series
$$
\widetilde{E}_k(\tau):=\begin{cases}
 1 \ \ \ \ \ &{\text {\rm if}}\ k\equiv 0\pmod{12},\\
 E_{4}(z)^2E_6(z) \ \ \ \ \ &{\text {\rm if}}\ k\equiv 2\pmod{12},\\
 E_4(z)\ \ \ \ \ &{\text {\rm if}}\ k\equiv 4\pmod{12},\\
 E_6(z)\ \ \ \ \ &{\text {\rm if}}\ k\equiv 6\pmod{12},\\
 E_4(z)^2\ \ \ \ \ &{\text {\rm if}}\ k\equiv 8\pmod{12},\\
E_{4}(z)E_6(z) \ \ \ \ \ &{\text {\rm if}}\ k\equiv 10\pmod{12}.
\end{cases}
$$
We use these Eisenstein series, Ramanujan's weight 12 cusp form
$$
\Delta(\tau)=\sum_{n=1}^{\infty} \tau(n)q^n:=\frac{E_4(\tau)^3-E_6(\tau)^2}{1728}=q-24q^2+252q^3-\cdots,
$$
and the modular function
$$
j(\tau):=\frac{E_4(\tau)^3}{\Delta(\tau)}=q^{-1}+744+196884q+\cdots
$$ 
to define a ``divisor polynomial'' for $f(\tau)$.
Namely, there is a polynomial $\widetilde{F}(f;x)$ (for example, see Lemma 2.34 of \cite{OnoCBMS}) for which
\begin{equation}
\widetilde{F}(f; j(\tau))= \frac{f(\tau)}{\Delta(\tau)^{m(k)}\widetilde{E}_k(\tau)},
\end{equation}
where 
$$
m(k):=\begin{cases}
 \lfloor k/12\rfloor \ \ \ \ \ &{\text {if}}\ k\not \equiv 2\pmod{12},\\
           \lfloor k/12\rfloor -1 \ \ \ \ \ &{\text {if}}\ k\equiv 2\pmod{12}.
\end{cases}
$$
Indeed, as the $\widetilde{E}_k(\tau)$ capture the trivial zeros at $\tau \in\{ i, \omega\},$ this expression is holomorphic on $\H$. In other words, it is a weakly holomorphic modular function (i.e. poles are supported the cusp infinity). As the $j$-function defines a bijection between the fundamental domain $\FF$ of $\SL_2(\Z)$ with $\C$, it follows that $\widetilde{F}(f; j(\tau))$ is a polynomial in $j(\tau)$.
After noting that $j(i)=1728$ and $j(\omega)=0,$
and then letting
$$
h_k(x):=\begin{cases}
  1 \ \ \ \ \ &{\text {\rm if}}\ k\equiv 0\pmod{12},\\
      x^2(x-1728) \ \ \ \ \ &{\text {\rm if}}\ k\equiv 2\pmod{12},\\
           x \ \ \ \ \ &{\text {\rm if}}\ k\equiv 4\pmod{12},\\
         x-1728 \ \ \ \ \ &{\text {\rm if}}\ k\equiv 6\pmod{12},\\
         x^2 \ \ \ \ \ &{\text {\rm if}}\ k\equiv 8\pmod{12},\\
               x(x-1728) \ \ \ \ \ &{\text {\rm if}}\ k\equiv 10\pmod{12},
 \end{cases}
$$
\noindent
we define the {\it divisor polynomial} of $f(\tau)$  by
 \begin{equation}
 F(f;x):= h_k(x)\cdot \widetilde{F}(f;x).
 \end{equation}
 
 As we have $j : \mathcal{A}\mapsto [0, 1728]$ (see (\ref{A})), 
 the Rankin-Swinnerton-Dyer theorem asserts that the zeros of $F(E_{2k};x)$ are in the interval $[0, 1728].$ 
Their paper has been generalized. R. A. Rankin, F. K. C. Rankin's father, proved \cite{Rankin} analogous results for certain meromorphic Poincar\'e series on $\SL_2(\Z).$

T. Asai, M. Kaneko, and H. Ninomiya \cite{AKN} discovered the same phenomenon in connection with the Hecke operators $T_0(m)$ acting on the modular $j$-function.
To describe their result, we recall an important sequence of modular functions.
Let $j_0(\tau):=1$, and for every positive integer
$m$, let $j_m(\tau)$ be the unique modular function
which is holomorphic on $\H$ whose $q$-expansion satisfies
\begin{equation} \label{jmPrincipal}
j_m(\tau)=q^{-m}+\sum_{n=1}^{\infty}c_m(n)q^n.
\end{equation}
Each $j_m(\tau)$ is a monic degree $m$ integer polynomial $\psi_m(x),$  with $x=j(\tau)$.
The first few  are:
\begin{displaymath}
\begin{split}
&\psi_0(j(\tau))=j_0(\tau)=1,\\
&\psi_1(j(\tau))=j_1(\tau)=j(\tau)-744=q^{-1}+196884q+\cdots,\\
&\psi_2(j(\tau))=j_2(\tau)=j(\tau)^2-1488j(\tau)+159768=q^{-2}+42987520q+\cdots,\\
&\psi_3(j(\tau))=j_3(\tau)=j(\tau)^3-2232j(\tau)^2+1069956j(\tau)-36866976=q^{-3}+2592899910q+\cdots.
\end{split}
\end{displaymath}
We refer to these expressions as ``Hecke  polynomials'', as they satisfy
\begin{equation}\label{Hecke_jm}
\psi_m(j(\tau))=j_m(\tau)= m(j_1(\tau) \ | \ T_0(m)).
\end{equation}
Asai, Kaneko, and Ninomiya \cite{AKN} proved that the zeros of $\psi_m(x)\in \Z[x]$ are simple and lie in $(0, 1728).$

\begin{remark}
The $\psi_m(x)$ are easily computed from the generating function
(see \cite{AKN, Faber1, Faber2})
$$
\sum_{m=0}^{\infty}\psi_m(x)q^m:=
\frac{E_4(\tau)^2E_6(\tau)}{\Delta(\tau)}
\cdot \frac{1}{j(\tau)-x}=1+(x-744)q+(x^2-1488x+159768)q^2+\cdots,
$$
which is equivalent to 
 the famous
denominator formula for the Monster Lie algebra (note. $p:=e^{2\pi i z}$)
$$
j(\tau)-j(z)=p^{-1}\prod_{m> 0 \ {\text {\rm and}}\
n\in \Z}(1-p^mq^n)^{c_1(mn)}.
$$
\end{remark}

Here we consider another sequence of Hecke polynomials. Instead of examining the Hecke action on furthermore classical modular forms and functions, we consider a special weight $-10$ {\it harmonic Maass form}, a nonholomorphic modular form $R(\tau)$ which for all $\left (\begin{smallmatrix} a&b\\c&d\end{smallmatrix}\right) \in \SL_2(\Z)$ and $\tau\in \H$ satisfies
$$
R\left(\frac{a\tau+b}{c\tau+d}\right) =(c\tau+d)^{-10}\cdot R(\tau),
$$
as well as
$$
	\Delta_{-10}(R(\tau)) = 0,
$$
where $\tau=u+iv$ and
$$
\Delta_k := -v^2\left( \frac{\partial^2}{\partial u^2} +
\frac{\partial^2}{\partial v^2}\right) + ikv\left(
\frac{\partial}{\partial u}+i \frac{\partial}{\partial v}\right).
$$
is the weight $k$ hyperbolic Laplacian operator.
For background on harmonic Maass  forms and their applications, the reader can see \cite{BOAnnals, BOInventiones,
BOPNAS, BOR, Br, BF, BruinierOno, OnoCDM, OnoMock}. 

Ramanujan's mock theta functions offer the first examples of
such
Maass forms
(for example, see \cite{ BOAnnals, BOInventiones, Zwegers1}).
For context, we recall the famous weight 1/2 harmonic Maass form
\begin{equation}\label{fq}
q^{-1}f(q^{24})+ N_f(\tau),
\end{equation}
where
$$
N_f(\tau):=2i\sqrt{3} \int_{-24\overline{\tau}}^{i\infty}
\frac{\sum_{n=-\infty}^{\infty}\left(n+\frac{1}{6}\right)e^{3\pi i
\left( n+\frac{1}{6}\right)^2 z }}{\sqrt{-i(z+24\tau)}}\ dz
$$
is a period integral of a weight 3/2 theta function, and $f(q)$ is Ramanujan's mock theta
function
$$
f(q):=1+\sum_{n=1}^{\infty}\frac{q^{n^2}}{(1+q)^2(1+q^2)^2\cdots
(1+q^n)^2}.
$$
Period integrals of modular forms are non-holomorphic, and so, following Ramanujan, the holomorphic parts of these harmonic Maass forms are known as
 {\it mock modular forms}. 
 
 In analogy with (\ref{fq}),  the second author used the method of Poincar\'e series to construct \cite{OnoMock} the weight $-10$ harmonic Maass form
\begin{equation}\label{MockDeltaDef}
R(\tau):=M_{\Delta}(\tau)+N_{\Delta}(\tau),
\end{equation}
where the theta function defining $N_f(\tau)$ is suitably replaced by Ramanujan's $\Delta(\tau)$, producing the period integral function
 \begin{equation}\label{periodintegral}
N_{\Delta}(\tau)=(2\pi)^{11}\cdot 11i\cdot \beta_{\Delta}\int_{-\overline{\tau}}^{i \infty}
\frac{\overline{\Delta(-\overline{z})}}{(-i(z+\tau))^{-10}} \ dz.
\end{equation}
Here $\beta_{\Delta}\sim 2.840287\dots$ is an explicit real number arising from the Poincar\'e series that represents $R(\tau)$ (see the Remark after Lemma~\ref{Poincare}).
The mock modular form
\begin{equation} \label{MDelta}
\begin{split}
M_{\Delta}(\tau)&=\sum_{n=-1}^{\infty}\ad(n)q^n=39916800q^{-1}-\frac{2615348736000}{691}-\cdots\\
&=11!\cdot q^{-1}+\frac{24\cdot 11!}{B_{12}}-73562460235.68364\dots q
-929026615019.11308\dots q^2\\
&\ \ \ \ \ \ \ \ \ \ \ \ \ \ -8982427958440.32917\dots q^3
-71877619168847.70781\dots q^4-\cdots
\end{split}
\end{equation}
has coefficients that are given as absolutely convergent infinite series.
Namely, for  $n\geq 1,$ we have
\begin{equation} \label{adFormula}
\ad(n)=-2\pi\Gamma(12)n^{-\frac{11}{2}}\cdot
                      \sum_{c=1}^{\infty}\frac{K(-1,n,c)}{c}\cdot I_{11}\left(\frac{4\pi \sqrt{n}}{c}
                      \right),
\end{equation}
where $I_{11}(x)$ is the usual $I_{11}$-Bessel function and $K(m,n,c)$ is the Kloosterman sum
\begin{equation} \label{Kloosterman}
K(m,n,c):=
\sum_{v(c)^{\times}} e\left(\frac{m\overline
v+nv}{c}\right),
\end{equation}
where $v$ runs over the primitive residue classes modulo $c$, and $\overline v$ is the multiplicative inverse of $v$.

The following theorem was proved by the second author in \cite{OnoMock}.

\begin{theorem}{\text {\rm {[Theorem 2.1 of \cite{OnoMock}]}}}\label{Theorem0}\newline 
The function $R(\tau)$ defined in (\ref{MockDeltaDef}) is a weight $-10$ harmonic Maass form on $\SL_2(\Z)$.
\end{theorem}

Although the spaces of harmonic Maass forms are infinite-dimensional, the Hecke operators continue to preserve these spaces. Therefore, in analogy with (\ref{Hecke_jm}), we can define Hecke polynomials that arise from the mock modular form $M_{\Delta}(\tau).$ Furthermore, we give explicit formulas for them in terms of the polynomials $\psi_m(x)$ discussed earlier.

\begin{theorem}\label{Theorem1}  If $m\geq 2$ is an integer, then the following are true.

\noindent
(1)  There is a monic degree $m$ polynomial $F_m(x)\in \Z[x]$ for which
 $$
  F_m(j(\tau))=\frac{E_4(\tau)E_6(\tau)}{11!}\cdot \left( m^{11}M_{\Delta}(\tau) \ | \ T_{-10}(m) - \tau(m)M_{\Delta}(\tau)\right).
$$

\noindent
(2) In terms of the Hecke polynomials for $j(\tau),$ we have 
\begin{equation*}
\begin{split}
F_m(x) =\psi_m(x) &+\frac{247944\tau(m)}{691}-\frac{65520\sigma_{11}(m)}{691} - 264\sigma_9(m)\\
& - 264 \sum_{l=1}^{m-2} \sigma_9(l) \psi_{m-l}(x) -(\tau(m) + 264\sigma_9(m-1))\psi_1(x).
\end{split}
\end{equation*}
\end{theorem}
\begin{remark} That $F_m(x)$ is a monic  degree $m$ polynomial is seen by the $\psi_m(x)$ term in Theorem~\ref{Theorem1} (2).
\end{remark}

\begin{examples}
For $2 \leq m \leq 5$, we offer the factorizations of these polynomials:
\begin{displaymath}
\begin{split}
F_2(x)&=x(x-1728),\\
F_3(x)&=x(x-768)(x-1728),\\
F_4(x)&=x(x^2 - 1512x + 374760)(x-1728),\\
F_5(x)&=x(x^3 - 2256x^2 + 1302804x - 149109760)(x-1728).
\end{split}
\end{displaymath}
\end{examples}

We establish that the zeros of the $F_m(x)$ exhibit similar behavior to those of the Eisenstein series, as in the work of Rankin and Swinnerton-Dyer \cite{RSD}, and the $j_m(\tau)$ modular functions, as revealed in the work of Asai, Kaneko, and Ninomiya \cite{AKN}.  Namely, we prove the following theorem.

\begin{theorem}\label{Theorem2}
 The following are true.

\noindent
(1) If $m\geq 2,$ then we have that $F_m(0)=F_m(1728)=0.$

\noindent
(2) If $m\geq 2,$ then the zeros of $F_m(x)$ are simple and are in $[0, 1728].$

\noindent
(3) As $m\rightarrow +\infty$, the zeros of $F_m(x)$ become equidistributed in $[0, 1728].$
\end{theorem}

We prove Theorem~\ref{Theorem1} by making use of the theory of Hecke operators for  harmonic Maass forms. We prove Theorem~\ref{Theorem2} by adapting the method of Rankin on Poincar\'e series, combined with the approach of  Asai, Kaneko, and Ninomiya on $q$-expansions of the Hecke images of the mock modular form $M_{\Delta}(\tau).$

\section*{Acknowledgements}
\noindent  The second author is grateful for the support of the Thomas Jefferson Fund, the NSF (DMS-2002265 and DMS-2055118), and the Simons Foundation (SFI-MPS-TSM-00013279).

\section{Proofs of Theorems~\ref{Theorem1} and \ref{Theorem2}}\label{Proofs}

Here we prove of our main results. In the first subsection we prove Theorem~\ref{Theorem1}, and in the following subsection we prove Theorem~\ref{Theorem2}.

\subsection{Proof of Theorem~\ref{Theorem1}}
These results rely on the delicate interplay between the action of the Hecke operators on $R(\tau)$, the weight $-10$ harmonic Maass form, and the weight 12 cusp form $\Delta(\tau)$. To this end,  we begin by recalling the action of the Hecke operators on these modular forms. For convenience, we let $\tau := u + iv$ throughout.
The action of the $m$th weight $k$ Hecke operator on weight $k$ modular forms, which preserves modularity, is given by
\begin{equation} \label{HeckeCoset}
	(f \ | \ T_k(m))(\tau) = m^{k-1} \sum_{\substack{ad=m \\ a,d>0}} \frac{1}{d^k} \sum_{b \bmod d} f\left(\frac{a\tau + b}{d}\right).
\end{equation}
This definition applies to harmonic Maass forms and classical meromorphic modular forms. Moreover, the Hecke action is easily described as combinatorial operators on Fourier expansions. In the case of harmonic Maass forms, extra care is required when the functions are not meromorphic.

To make this precise, we recall that if
 $f(\tau)$ is a harmonic Maass form of even weight $k$, then it has a Fourier expansion of the form (for example, see Lemma 4.2 of \cite{HMF})
\begin{equation} \label{qExp}
	f(\tau) = \sum_{n \gg -\infty} c_f^+(n)q^n + \sum_{n < 0} c_f^-(n) \Gamma(1-k,-4\pi n v)q^n,
\end{equation}
where
$$
	\Gamma(s,z) := \int_z^\infty e^{-t}t^s \frac{dt}{t}
$$
is the incomplete gamma function. The $q$-series
$$
f^{+}(\tau):=\sum_{n\gg -\infty} c_f^+(n)q^n
$$
is the {\it holomorphic part} of $f(\tau)$, while its {\it nonholomorphic part} is
$$
f^-(\tau):= \sum_{n < 0} c_f^-(n) \Gamma(1-k,-4\pi n v)q^n.
$$
In the case of $R(\tau)$, we have $f^+(\tau) = M_{\Delta}(\tau)$ and $f^-(\tau) = N_{\Delta}(\tau)$. Furthermore, for meromorphic modular forms $f(\tau)$, we have that $f(\tau)=f^+(\tau)$.

The following proposition describes the action of the Hecke operators on even integer weight meromorphic modular forms and harmonic Maass forms on $\SL_2(\Z).$

\begin{proposition} \label{HeckeFacts} Suppose that $f(\tau)$ is an even integer weight $k$ modular form on $\SL_2(\Z).$
	If $m \geq 1$ is an integer, then the following are true.
	
	\noindent
	(1) If $f(\tau)$ is a meromorphic modular form, then we have
		$$
			(f \ | \ T_k(m))(\tau) = \sum_{n \gg -\infty} \sum_{\substack{d \mid (m,n) \\ d > 0}} d^{k-1} c^+_f\left(\frac{mn}{d^2}\right)q^n.
		$$

		\noindent
		(2)  If $f(\tau)$ is a harmonic Maass form, then we have
		\begin{equation*}
		\begin{split}
			(f \ | \ T_k(m))(\tau) &= \sum_{n \gg -\infty} \sum_{\substack{d \mid (m,n) \\ d > 0}} d^{k-1} c_f^+\left(\frac{mn}{d^2}\right)q^n \\ &+ \sum_{n < 0} \sum_{\substack{d \mid (m,n) \\ d > 0}} d^{k-1} c_f^-\left(\frac{mn}{d^2}\right)\Gamma(1 - k,-4\pi n(d^2/m)v)q^n.
		\end{split}
		\end{equation*}
\end{proposition}

The proof of Theorem~\ref{Theorem1} relies on the fact that $R(\tau)$ was constructed so that the weight $-10$ Hecke action on its nonholomorphic part is essentially the weight 12 Hecke action on $\Delta(\tau)$. The special sequence of polynomials $F_m(x)$ then arises from suitable linear combinations of these forms in which the nonholomorphic parts cancel. This principle is conceptually captured through the work of Bruinier and Funke (see Proposition 3.2 of \cite{BF} and Theorem 5.5 of \cite{HMF}) by the properties of the differential operator
$$
	\xi_k := 2iv^k \overline{\frac{\partial}{\partial \overline{\tau}}},
$$
whose action on the Fourier expansion of a harmonic Maass form of weight $k$ is given explicitly (see Theorem 5.9 of \cite{HMF}) by
\begin{equation} \label{Xi}
	\xi_k(f(\tau)) = \xi_k(f^-(\tau)) = -(4\pi)^{1-k}\sum_{n=1}^{\infty} \overline{c_f^-(-n)} n^{1-k}q^n.
\end{equation}
Furthermore, $\xi_k(f^-(\tau))$ is a holomorphic cusp form weight $2 - k$, which we refer to as the \emph{shadow} of $f^+$.  In the present case, we have $f(\tau)=R(\tau)$ and $\xi_{-10}(f(\tau)) = -11\beta_{\Delta} \cdot \Delta(\tau)$. More generally, Proposition~\ref{HeckeFacts} and \eqref{Xi} yield the following lemma concerning the action of the Hecke operators on the shadow of a mock modular form.

\begin{lemma} \label{XiHecke}
If $f(\tau)$ is a harmonic Maass form of even weight $k < 0$ on $\SL_2(\Z)$ and $m \geq 2,$ then we have
	$$
		m^{1-k}\xi_k(f^-(\tau) \ | \ T_k(m)) =  \xi_k(f^-(\tau)) \ | \ T_{2-k}(m).
	$$
\end{lemma}
\begin{proof}
	We first write
	$$
		\xi_k(f(\tau)) = \xi_k(f^+(\tau)) + \xi_k(f^-(\tau)).
	$$
	Since $f^+(\tau)$ is holomorphic, it satisfies the Cauchy-Riemann equations, so we see from the definition of $\xi_k$ that $\xi_k(f^+(\tau))$ vanishes.  Using \eqref{HeckeCoset} and the chain rule, we obtain
	\begin{align*}
		m^{1-k}\xi_k(f^{-}(\tau) \ | \ T_k(m)) &= \sum_{\substack{ad=m \\ a,d>0}} \frac{1}{d^k} \sum_{b \bmod d} \xi_k\left( f^-\left(\frac{a\tau + b}{d}\right)\right) \\
		&= \sum_{\substack{ad=m \\ a,d>0}} \frac{1}{d^k} \sum_{b \bmod d} 2iv^k \left( \frac{a}{d} \right)\left( \frac{\partial \overline{f^-}}{\partial \tau} \right) \left( \frac{a\tau + b}{d} \right).
	\end{align*}

	Rearranging, we find that
	\begin{align*}
		m^{1-k}\xi_k(f^{-}(\tau) \ | \ T_k(m)) &= \sum_{\substack{ad=m \\ a,d>0}} \frac{1}{d^k} \left(\frac{d}{a}\right)^k \left( \frac{a}{d} \right) \sum_{b \bmod d} 2i\left( \frac{av}{d} \right)^k\left( \frac{\partial \overline{f^-}}{\partial \tau} \right) \left( \frac{a\tau + b}{d} \right) \\
		&= m^{1-k}\sum_{\substack{ad=m \\ a,d>0}} \frac{1}{d^{2-k}} \sum_{b \bmod d} (\xi_k(f^-))\left(\frac{a\tau + b}{d}\right) \\
		&= \xi_k(f^-(\tau)) \ | \ T_{2-k}(m).
	\end{align*}
\end{proof}

With this background information, we are able to prove Theorem~\ref{Theorem1}.

\begin{proof}[Proof of Theorem~\ref{Theorem1}]
First we prove (1). By Lemma 5.16 of \cite{HMF} applied to \eqref{periodintegral}, we have that
$$
	\xi_{-10}(R(\tau)) = \xi_{-10}(N_{\Delta}(\tau)) = -11\beta_{\Delta} \cdot \Delta(\tau).
$$
For each $m \geq 1$, $\Delta(\tau)$ is an eigenform of $T_{12}(m)$ with eigenvalue $\tau(m)$, whereby Lemma~\ref{XiHecke} gives
$$
	m^{11}\xi_{-10}(N_{\Delta}(\tau) \ | \ T_{-10}(m)) = \xi_{-10}(N_{\Delta}(\tau)) \ | \ T_{12}(m) = -11\beta_{\Delta} \cdot \tau(m)\Delta(\tau).
$$
Thus, $N_{\Delta}(\tau)$ is an eigenform of $T_{-10}(m)$ with eigenvalue $m^{-11}\tau(m)$. This implies that
$$
	m^{11}N_{\Delta}(\tau) \ | \ T_{-10}(m) - \tau(m)N_{\Delta}(\tau) = 0.
$$
Therefore, we find that
$$
	m^{11}R(\tau) \ | \ T_{-10}(m) - \tau(m)R(\tau) = m^{11}M_{\Delta}(\tau) \ | \ T_{-10}(m) - \tau(m)M_{\Delta}(\tau)
$$
is a weakly holomorphic modular form on $\SL_2(\Z)$. Since $R(\tau)$ is of weight $-10$, we then have that
$$
	E_4(\tau)E_6(\tau)(m^{11}M_{\Delta}(\tau) \ | \ T_{-10}(m) - \tau(m)M_{\Delta}(\tau))
$$
is a weakly holomorphic modular function, which is in turn a polynomial in $j(\tau)$. Furthermore, its leading term is $11! \cdot q^{-m}$ by Proposition~\ref{HeckeFacts} (1) and \eqref{MDelta}, which implies that $F_m(x)$ is monic of degree $m$ as claimed.

To derive the formula in claim (2), we employ Proposition~\ref{HeckeFacts} (2), which gives
\begin{equation} \label{Hecke10}
	m^{11}M_{\Delta}(\tau) \ | \ T_{-10}(m) = m^{11} \sum_{n \leq 0} \sum_{\substack{d \mid (m,n) \\ d > 0}} d^{-11} \ad\left( \frac{mn}{d^2}\right)q^n + O(q).
\end{equation}
Since $a_{\Delta}(mn/d^2) = 0$ whenever $mn/d^2 \leq -2$, we need only identify values of $n$ and $d$ for which $mn/d^2 \in \{-1,0\}$. If $mn/d^2 = 0$, then $n = 0$ and we obtain the consterm term of \eqref{Hecke10}:
$$
	m^{11} \sum_{d \mid m} d^{-11} a_{\Delta}(0) = \sum_{d \mid m} (m/d)^{11} \ad(0) = \ad(0)\sigma_{11}(m).
$$
If $mn/d^2 = -1$, then $n = -d^2/m$. Since $d \mid n$, we must have $d \mid d^2/m$, whereby $m \mid d$ and $d = m$. Altogether, we have
$$
	m^{11} M_{\Delta}(\tau) \ | \ T_{-10}(m) = \ad(-1) q^{-m} + \ad(0)\sigma_{11}(m) + O(q).
$$

We then multiply by $E_4(\tau)E_6(\tau) = E_{10}(\tau)$, making use of the Fourier expansion
$$
	E_{10}(\tau) = 1 - 264\sum_{n=1}^{\infty} \sigma_9(n)q^n
$$
to obtain
\begin{equation} \label{FmHalf}
	\begin{split}
		E_4(\tau)E_6(\tau) \cdot (m^{11} M_{\Delta}(\tau) \ | \ T_{-10}(m)) &= \ad(-1)\left(q^{-m} - \sum_{l=1}^{m} \sigma_9(l)q^{-m+l}\right) \\ &\ \ \ \ \ + \ad(0)\sigma_{11}(m) + O(q).
	\end{split}
\end{equation}
Likewise, we compute
$$
	E_4(\tau)E_6(\tau) \cdot\left( -\tau(m)M_{\Delta}(\tau)\right) = -\tau(m)(\ad(-1)q^{-1} - 264\ad(-1) + \ad(0)) + O(q),
$$
which together with \eqref{FmHalf} gives
\begin{equation*}
	\begin{split}
		11! \cdot F_m(j(\tau)) &= \ad(-1)\left(q^{-m} - \sum_{l=1}^{m} \sigma_9(l)q^{-m+l}\right) \\ &\ \ \ \ \ + \ad(0)\sigma_{11}(m) - \tau(m)(\ad(-1)q^{-1} - 264\ad(-1) + \ad(0)) + O(q).
	\end{split}
\end{equation*}

We finally recall that $\{j_0(\tau),j_1(\tau),\dots\}$ is a polynomial basis for all modular functions on $\SL_2(\Z)$. Thus, thanks to \eqref{jmPrincipal} and \eqref{Hecke_jm}, we have
\begin{equation*}
	\begin{split}
		11! \cdot F_m(j(\tau)) &= \ad(-1)\left(\psi_{m}(j(\tau)) - \sum_{l=1}^{m} \sigma_9(l)\psi_{m-l}(j(\tau))\right) \\ &\ \ \ \ \ + \ad(0)\sigma_{11}(m) - \tau(m)(\ad(-1)\psi_1(j(\tau)) - 264\ad(-1) + \ad(0)).
	\end{split}
\end{equation*}
Substituting the values of $\ad(-1)$ and $\ad(0)$ from \eqref{MDelta} gives the formula in Theorem~\ref{Theorem1} (2).
\end{proof}

\subsection{Proof of Theorem~\ref{Theorem2}}

The proof of Theorem~\ref{Theorem2} relies on the following infinite sequence of Poincar\'e series (see \cite{HMF}, \cite{OnoCBMS}). For non-positive integer weights $k$ and positive integers $m$, we define
\begin{equation}\label{PoincareDefinition}
	\cP_{k,-m}(\tau) := \sum_{M \in \Gamma_{\infty} \backslash \SL_2(\Z)} (\phi_{-m}\ |_k \ M)(\tau), 
\end{equation}
where  $\Gamma_{\infty} := \{\pm \left(\begin{smallmatrix}1 & n \\ 0 & 1\end{smallmatrix}\right), n \in \Z\}$ is the group of translations, and where
$$
	\phi_{-m}(\tau) := (4 \pi m v)^{-\frac{k}{2}} M_{-\frac{k}{2},\frac{1-k}{2}}(4\pi m v) e^{-2\pi i m u},
$$
and $M_{\kappa,\mu}$ is the usual $M$-Whittaker function. Here we recall the basic properties of these Poincar\'e series as harmonic Maass forms, which includes their Fourier expansions, that are given in terms of $I$-Bessel and $J$-Bessel functions and the Kloosterman sums defined in \eqref{Kloosterman}.

\begin{lemma}[Theorem 6.10 of \cite{HMF}] \label{MaassPoincare}
	Assuming the hypotheses above, the following are true.
	
	\noindent
	(1) We have that $\cP_{k,-m}(\tau)$ is a weight $k$ harmonic Maass form on $\SL_2(\Z)$, with
		$$
		\cP_{k,-m}(\tau) = 	\cP_{k,-m}^+(\tau) + 	\cP_{k,-m}^-(\tau),
		$$
		where $\cP_{k,-m}^+(\tau)$ (resp. $\cP_{k,-m}^-(\tau)$) is its holomorphic part (resp. nonholomorphic part).
		
		\smallskip
		\noindent
		(2)	 The holomorphic part of $\cP_{k,-m}(\tau)$ has Fourier expansion
		$$
		\cP_{k,-m}^+(\tau)= q^{-m} - \frac{(2\pi i)^{2-k}m^{1-k}}{1 - k}\sum_{c > 0} \frac{K_k(m,0;c)}{c^{2-k}} + \sum_{n=1}^{\infty} c^+_{k,-m}(n)q^n,
		$$
		where for positive integers $n$ we have
		$$
			c_{k,-m}^+(n) = 2\pi i^k \left(\frac{m}{n}\right)^{\frac{1-k}{2}}  \sum_{c > 0} \frac{K_k(m,n;c)}{c} \cdot I_{1-k}\left(\frac{4\pi\sqrt{mn}}{c}\right).
		$$	
		
		\smallskip
		\noindent
		(3) The nonholomorphic part of $\cP_{k,-m}(\tau)$ has Fourier expansion
		$$
		\cP_{k,-m}^-(\tau) = \sum_{n=1}^{\infty} c_{k,-m}^-(-n)\Gamma(1 - k, 4\pi n v)q^{-n},	
		$$
		where for positive integers $n$ we have
		$$
			c_{k,-m}^-(-n) = 2\pi i^k \left(\frac{m}{n}\right)^{\frac{1-k}{2}}  \sum_{c > 0} \frac{K_k(m,-n;c)}{c} \cdot J_{1-k}\left(\frac{4\pi\sqrt{mn}}{c}\right).
		$$
\end{lemma}

Relevant for this work, we now relate $R(\tau)$ to these Poincar\'e series, and we explicitly describe the Hecke action on these harmonic Maass forms.

\begin{lemma} \label{Poincare}
	Assuming the notation above, the following are true.

	\noindent
	(1) We have that $R(\tau)= \cP_{-10,-1}(\tau).$
	
	\noindent
	(2)	For all $m \geq 2$, $R(\tau) \ | \ T_{-10}(m) = m^{-11} \cdot \cP_{-10,-m}(\tau)$.
\end{lemma}

\begin{remark}
We note that the real number $\beta_{\Delta}=2.840287\dots$ is
$$
\beta_{\Delta}:=	1 + c_{-10,-1}^-(-1) = 1 + 2\pi \sum_{c > 0} \frac{K_k(1,1;c)}{c} \cdot J_{11}\left( \frac{4\pi}{c} \right).
$$
This is the coefficient of $q^1$ of the cuspidal weight $12$ Poincar\'e series
$$
	\sum_{M \in \Gamma_{\infty} \backslash \SL_2(\Z)} (\widehat{\varphi}_{1} \ |_{12} \ M)(\tau),
$$
where $\widehat{\varphi}_{m}(\tau) = e^{2\pi i m\tau}$.
\end{remark}

\begin{proof}
The proof of Theorem~\ref{Theorem0} (see \cite{OnoMock}) includes the derivation of (1). To prove (2), we recall from the proof of Theorem~\ref{Theorem1} that
	$$
		M_{\Delta}(\tau) \ | \ T_{-10}(m) = 11! \cdot m^{-11}q^{-m} + O(1). 
	$$
	Meanwhile, by Lemma~\ref{MaassPoincare} (2), we have that
	$$
		\cP^+_{-10,-m} = 11! \cdot q^{-m} + O(1).
	$$
	 Therefore, $m^{-11}\cdot\cP^+_{-10,-m}(\tau) - M_{\Delta}(\tau) \ | \ T_{-10}(m) = O(1)$, and Lemma 5.12 of \cite{HMF} thus implies that $m^{-11}\cdot\cP_{-10,-m}(\tau) - R(\tau) \ | \ T_{-10}(m)$ is either identically zero or a weight $-10$ holomorphic form. Since no such forms exist, we obtain claim (2). 
\end{proof}

Thanks to this lemma, the proof of Theorem~\ref{Theorem2} is reduced to the study of the zeros of nonholomorphic Poincar\'e series, which is analogous to the work of Rankin \cite{Rankin}. Here we suitably adapt his approach. For the sequel, we let $\tau = e^{i\theta},$ where $\theta \in [\frac{\pi}{3},\frac{\pi}{2}],$ so that $\tau \in \cA$. If $f$ is a weakly holomorphic form of weight $k$ with real Fourier coefficients, then $f(-1/e^{i\theta}) = \overline{f(e^{i\theta})}$, wherein
$$
	\overline{e^{\frac{k}{2}i\theta}f(e^{i\theta})} = 	e^{-\frac{k}{2}i\theta}\overline{f(e^{i\theta})} = e^{-\frac{k}{2}i\theta}f(-1/e^{i\theta}) = e^{\frac{k}{2}i\theta}f(e^{i\theta})
$$
by modularity, implying that $e^{\frac{k}{2}i\theta}f(e^{i\theta})$ is real (see Proposition 2.1 of \cite{Getz}). For our case of $M_{\Delta}(\tau)$, we will first show that $e^{-5i\theta}(m^{11}R(e^{i\theta}) \ | \ T_{-10}(m) - \tau(m)R(e^{i\theta}))$ is well-approximated by a damped cosine wave $f_m(\theta)$ which controls the location of its zeros, and hence the zeros of $F_m(j(e^{i\theta}))$. In particular, we will show that
\begin{equation} \label{Goal}
	\abs{e^{-5i\theta}e^{-2\pi m \sin \theta}\left(m^{11}R(e^{i\theta}) \ | \ T_{-10}(m) - \tau(m)R(e^{i\theta})\right) - f_m(\theta)} < 11!
\end{equation}
for all $m \geq 3$, where
\begin{equation}
	f_m(\theta) := 2 \cdot 11! \cdot \left(1 - e^{-4\pi m \sin \theta }e_{10}(4\pi m \sin \theta)\right)\cos\left(5\theta + 2\pi m \cos \theta \right)
\end{equation}
and
$$
	e_j(x) := \sum_{n=0}^{j} \frac{x^n}{n!}
$$
is the $j$th order Taylor approximation of $e^x$. By Lemma~\ref{Poincare}, this is equivalent to showing
\begin{equation*}
	\abs{e^{-5i\theta}e^{-2\pi m \sin \theta}\left(\cP_{-10,-m}(e^{i\theta}) - \tau(m)R(e^{i\theta})\right) - f_m(\theta)} < 11!
\end{equation*}
for all $m \geq 3$.

To this end, we employ (\ref{PoincareDefinition}) to obtain
$$
	\cP_{-10,-m}(e^{i\theta}) = \sum_{\substack{c \geq 0, d \in \Z \\ (c,d)=1}} (ce^{i\theta} + d)^{10} \phi_{-m}\left( \frac{ae^{i\theta} + b}{ce^{i\theta} + d} \right),
$$
where $a,b \in \Z$ is an arbitrary solution to $ad - bc = 1$. We extract the main terms corresponding to $(c,d) = (0,1)$ and $(1,0)$, obtaining
\begin{equation} \label{MainTerms}
	\cP_{-10,-m}(e^{i\theta}) = \phi_{-m}(e^{i\theta}) + e^{10i\theta}\phi_{-m}(-e^{-i\theta}) + G_{-m}(\theta),
\end{equation}
where
$$
	G_{-m}(\theta) := \phi_{-m}(e^{i\theta}) + \sum_{\substack{c \geq 1, d \in \Z \setminus\{0\}\\ (c,d)=1}} (ce^{i\theta} + d)^{10} \phi_{-m}\left( \frac{ae^{i\theta} + b}{ce^{i\theta} + d} \right).
$$

We rewrite the main terms in \eqref{MainTerms} using the definition of $\phi_{-m}$, which yields
\begin{equation} \label{MainTermsSimp}
	\phi_{-m}(e^{i\theta}) + e^{10i\theta}\phi_{-m}(-e^{-i\theta}) = (4\pi m\sin \theta)^5M_{5,\frac{11}{2}}(4\pi m \sin \theta)\left(e^{2\pi i m \cos \theta} + e^{10i\theta}e^{-2\pi i m \cos \theta}\right).
\end{equation}
By (13.18.4) and (8.4.7) of \cite{DLMF}, we have, for all $\kappa \in \N$ and $x \in \R$, that
\begin{equation} \label{MSimp}
	M_{\kappa,\kappa+\frac{1}{2}}(x) = (2\kappa + 1)! \cdot \frac{e^{\frac{x}{2}} - e^{-\frac{x}{2}}e_{2\kappa}(x)}{x^{\kappa}}.
\end{equation}
We multiply \eqref{MainTermsSimp} by $e^{-5i\theta}e^{-2\pi m \sin \theta}$ and apply \eqref{MSimp}, which simplifies to $f_m(\theta)$. Therefore, we seek to bound
\begin{equation} \label{AllTerms}
	\abs{e^{-5i\theta}e^{-2\pi m \sin \theta}\left(\cP_{-10,-m}(e^{i\theta}) - \tau(m)R(e^{i\theta})\right) - f_m(\theta)} = e^{-2\pi m \sin \theta}\abs{G_{-m}(\theta) - \tau(m)R(e^{i\theta})}.
\end{equation}

From here, we require a number of lemmas to establish \eqref{Goal}. We first obtain a useful bound for the $M$-Whittaker function on the positive real line.

\begin{lemma} \label{MBound}
	If $\kappa \in \N$ and $x \geq 0$, then we have
	$$
		M_{\kappa,\kappa + \frac{1}{2}}(x) \leq e^{\frac{x}{2}} x^{\kappa+1}.
	$$
\end{lemma}

\begin{proof}
	Rearranging \eqref{MSimp}, we have
	$$
		\frac{e^{\frac{x}{2}}}{(2\kappa + 1)!} \cdot M_{\kappa,\kappa+\frac{1}{2}}(x) = \frac{e^x - e_{2\kappa}(x)}{x^{\kappa}}.
	$$
	Since $e_{2\kappa}(x)$ is the $2\kappa$th Taylor expansion of $e^x$ and $e^x$ is increasing on the real line, we obtain via Taylor's theorem
	$$
		\frac{e^{\frac{x}{2}}}{(2\kappa + 1)!} \cdot M_{\kappa,\kappa+\frac{1}{2}}(x) = x^{-\kappa}\sum_{n=2\kappa+1}^{\infty} \frac{x^{n}}{n!} \leq \frac{e^x x^{\kappa+1}}{(2\kappa+1)!}.
	$$
	Isolating $M_{\kappa,\kappa+\frac{1}{2}}(x)$ yields our result.
\end{proof}

Next, we bound the terms $G_{-m}(\theta)$.

\begin{lemma} \label{GBound}
If $m \geq 1$ and $\theta \in [\frac{\pi}{3},\frac{\pi}{2}]$, then we have
	$$
		\abs{G_{-m}(\theta)} \leq 2092m^{-4}e^{2\pi m\sin \theta} +  1.27 \cdot 10^{10} m^6.
	$$
\end{lemma}

\begin{proof}
	By the triangle inequality, we have
	$$
		\abs{G_{-m}(\theta)} \leq \abs{\phi_{-m}(e^{i\theta})} + \sum_{\substack{c \geq 1, d \in \Z \setminus\{0\}\\ (c,d)=1}} \abs{ce^{i\theta} + d}^{10} \abs{\phi_{-m}\left( \frac{ae^{i\theta} + b}{ce^{i\theta} + d} \right)}.
	$$
	Substituting the definition of $\phi_{-m}$ yields
	$$
		\sum_{\substack{c \geq 1, d \in \Z \setminus\{0\}\\ (c,d)=1}} \abs{ce^{i\theta} + d}^{10} \abs{\phi_{-m}\left( \frac{ae^{i\theta} + b}{ce^{i\theta} + d} \right)} = (4\pi m \sin \theta)^5 \sum_{\substack{c \geq 1, d \in \Z \setminus\{0\}\\ (c,d)=1}} M_{5,\frac{11}{2}}\left(\frac{4\pi m \sin \theta}{\abs{ce^{i\theta} + d}^2}\right).
	$$
	Since $M_{5,\frac{11}{2}}$ is monotonically increasing on the real line, the identity $\abs{ce^{i\theta}+d}^2 = c^2 + 2cd\cos \theta + d^2$ and the fact that $\cos \theta \leq \frac{1}{2}$ for $\theta \in [\frac{\pi}{3},\frac{\pi}{2}]$, imply that
	$$
		\sum_{\substack{c \geq 2, d \in \Z \\ (c,d)=1}} M_{5,\frac{11}{2}}\left(\frac{4\pi m \sin \theta}{\abs{ce^{i\theta} + d}^2}\right) \leq \sum_{\substack{c \geq 2, d \in \Z \\ (c,d)=1}} M_{5,\frac{11}{2}}\left(\frac{4\pi m \sin \theta}{c^2 + cd + d^2}\right).
	$$
	
	We separate the right-hand sum into two parts. If $m \leq c^2 + cd + d^2$, then
	$$
		\frac{4\pi m \sin \theta}{c^2 + cd + d^2} \leq 4\pi \sin \theta \leq 4\pi,
	$$
	so Lemma~\ref{MBound} gives
	$$
		M_{5,\frac{11}{2}}\left(\frac{4\pi m \sin \theta}{c^2 + cd + d^2}\right) \leq e^{2\pi}(4\pi m)^6 \frac{1}{(c^2 + cd + d^2)^6}.
	$$
	Otherwise, by \eqref{MSimp}, we have $M_{5,\frac{11}{2}}(x) \leq 11! \cdot x^{-5}e^{\frac{x}{2}}$, which then implies
	$$
		M_{5,\frac{11}{2}}\left(\frac{4\pi m \sin \theta}{c^2 + cd + d^2}\right) \leq 11! \cdot (4 \pi m \sin \theta)^{-5}e^{2\pi m \sin \theta} \leq 11! \cdot \left(\frac{4\pi m\sqrt{3}}{2}\right)^{-5}e^{2\pi m \sin \theta}
	$$
	since $\sin \theta \geq \frac{\sqrt{3}}{2}$. Therefore, we conclude that
	$$
		\abs{G_{-m}(\theta)} \leq 11! \cdot \left(\frac{4\pi m\sqrt{3}}{2}\right)^{-5}\sum_{\substack{c \geq 1, d \in \Z \\ c^2 + cd + d^2 < m}} e^{2\pi m \sin \theta} + e^{2\pi}(4\pi m)^6\sum_{\substack{c \geq 1, d \in \Z \\ c^2 + cd + d^2 \geq m}} \frac{1}{(c^2 + cd + d^2)^6}.
	$$
	Since all pairs $(c,d)$ with $c^2 + cd + d^2 < m$ satisfy $\max\{\abs{c},\abs{d}\} < \sqrt{3m}$, there are at most $(\sqrt{3m} + 1)^2 \leq 8m$ terms in the left-hand sum. Meanwhile, the right-hand sum is bounded above by the Epstein zeta value
	$$
		\sum_{\substack{(c,d) \in \Z \\ (c,d) \neq (0,0)}} \frac{1}{(c^2 + cd + d^2)^6} \leq 6.0099.
	$$
	Putting this all together, we find that
	$$
		\abs{G_{-m}(\theta)} \leq 8 \cdot 11! \cdot \left(\frac{4\pi m \sqrt{3}}{2}\right)^{-5}m e^{2\pi m \sin \theta} + 6.0099e^{2\pi}(4\pi m)^6,
	$$
	which, upon simplification, gives the claimed result.
\end{proof}

Finally, we bound the contributions of the $\tau(m)R(e^{i\theta})$ term in (\ref{AllTerms}).

\begin{lemma} \label{Deligne}
	For all $m \geq 1$ and $\theta \in [\frac{\pi}{3},\frac{\pi}{2}]$, we have
	$$
		\abs{\tau(m)R(e^{i\theta})} \leq  4.20 \cdot 10^{11}m^6.
	$$
\end{lemma}

\begin{proof}
	Since $R(\tau) = \cP_{-10,-1}(\tau)$ by Lemma~\ref{Poincare} (1), we obtain via \eqref{MainTerms}
	$$
		\abs{R(e^{i\theta})} \leq \abs{\phi_{-1}(e^{i\theta})} + \abs{\phi_{-1}(-e^{-i\theta})} + \abs{G_{-1}(\theta)}. 
	$$
	The definition of $\phi_{-1}$ gives
	$$
		\abs{R(e^{i\theta})} \leq 2(4\pi \sin \theta)^5 M_{5,\frac{11}{2}}(4\pi \sin \theta) + \abs{G_{-1}(\theta)},
	$$
	wherein Lemmas~\ref{MBound} and \ref{GBound} yield
	$$
		\abs{R(e^{i\theta})} \leq 2(4\pi \sin \theta)^{10} + 2092e^{2\pi\sin \theta} +  1.27 \cdot 10^{10} \leq 2.10 \cdot 10^{11}.
	$$
	We then apply Deligne's Theorem \cite{Deligne} for Ramanujan's tau function, which states that
	$$
		\abs{\tau(m)} \leq \sigma_0(m)m^{\frac{11}{2}}.
	$$
	Since $\sigma_0(m) \leq 2\sqrt{m}$, we obtain $\abs{\tau(m)} \leq 2m^6$, which implies our result.
\end{proof}

We are now equipped to prove \eqref{Goal}.

\begin{lemma} \label{fBound}
	For all $m \geq 3$ and $\theta \in [\frac{\pi}{3},\frac{\pi}{2}]$, we have
	$$
		\abs{e^{-5i\theta}e^{-2\pi m \sin \theta}\left(\cP_{-10,-m}(\tau) - \tau(m)R(e^{i\theta})\right) - f_m(\theta)} < 11!.
	$$
\end{lemma}

\begin{proof}
	By \eqref{AllTerms}, we have
	$$
		\abs{e^{-5i\theta}e^{- 2\pi m \sin \theta}(\cP_{-10,-m}(\tau) - \tau(m)R(e^{i\theta})) - f_m(\theta)} \leq e^{-2\pi m\sin \theta}\left(\abs{G_{-m}(\theta)} + \abs{\tau(m)R(e^{i\theta})} \right).
	$$
	Applying Lemmas~\ref{GBound} and \ref{Deligne} yields
	$$
		\abs{G_{-m}(\theta)} + \abs{\tau(m)R(e^{i\theta})} \leq 2092 m^{-4} e^{2\pi m \sin \theta} +  4.33\cdot 10^{11} m^6.
	$$
	Since $\sin \theta \geq \frac{\sqrt{3}}{2}$ and $m\geq 3,$ direct calculation gives
	$$
		e^{-2\pi m\sin \theta}\left(\abs{G_{-m}(\theta)} + \abs{\tau(m)R(e^{i\theta})} \right) \leq 2092m^{-4} +  4.33\cdot 10^{11} m^6e^{-2\pi m\frac{\sqrt{3}}{2}} < 11!.
	$$
\end{proof}

\begin{proof}[Proof of Theorem~\ref{Theorem2}] We begin by noting that the zeros of $E_4(\tau)E_6(\tau)$ arise universally thanks to Theorem~\ref{Theorem1} (1). Therefore, claim (1) follows from the fact that
 $E_4(\omega) = E_6(i) = 0$ and $j(\omega)=0$ and $j(i)=1728.$
	
	We now turn to claim (2), which is immediate for $m = 2$ as $F_2(x)=x(x-1728).$  Therefore, we assume that $m \geq 3$. Let $g_m(\theta) := 5\theta + 2\pi m \cos \theta$, so that
	$$
		f_m(\theta) = 2 \cdot 11! \cdot (1 - e^{-4\pi m \sin \theta}e_{10}(4\pi m \sin \theta)) \cos(g_m(\theta)).
	$$
	The zeros and extrema of $f_m(\theta)$ are controlled by $g_m(\theta)$. One readily verifies by computation that
	$$
		1 - e^{-4\pi m \sin \theta}e_{10}(4\pi m \sin \theta) \geq 0.99,
	$$
	and thus $\abs{f_m(\theta)} \geq 1.98 \cdot 11!$ whenever $g_m(\theta)$ is an integer multiple of $\pi$.
	
	We then observe that $g_m(\theta)$ decreases from $\frac{5\pi}{3} + \pi m$ to $\frac{5\pi}{2}$ on $[\frac{\pi}{3},\frac{\pi}{2}]$, hitting $m - 2$ consecutive integer multiples of $\pi$. Thus, Lemma~\ref{fBound} implies that $m^{11}R(e^{i\theta}) \ | \ T_{-10}(m) - \tau(m)R(e^{i\theta})$ changes sign at least $m - 2$ times in $[\frac{\pi}{3},\frac{\pi}{2}]$, once in each subinterval of the form
	$$
		(g_m^{-1}(\pi (\ell + 1)),g_m^{-1}(\pi \ell))
	$$
	for $3 \leq \ell \leq m$. Therefore, by the Intermediate Value Theorem, $m^{11}R(\tau) \ | \ T_{-10}(m) - \tau(m)R(\tau)$ has a zero in each subinterval. By Theorem~\ref{Theorem1} (1), each of these zeros is a zero of $F_m(j(e^{i\theta}))$.
	
	Since $j \colon \cA \to [0, 1728]$ is a bijection, we have located in all $m$ simple zeros of $F_m(x)$ in $[0, 1728]$. As $F_m(x)$ has degree $m$, this is the entire set of zeros, proving (2). Equidistribution of the zeros as $m \to +\infty$ is then apparent as these intervals essentially describe $\cA$ as a disjoint union of subarcs of equal size.
\end{proof}


\begin{thebibliography}{99}


\bibitem{AKN} T. Asai, M. Kaneko and
H. Ninomiya,
\emph{Zeros of certain modular functions and an
application},
Comm. Math. Univ. Sancti Pauli
\textbf{46} (1997), pages 93-101.

\bibitem{HMF} K. Bringmann, A. Folsom, K. Ono, and L. Rolen, \emph{Harmonic Maass forms and mock modular forms: Theory and applications},
Amer. Math. Soc. Colloq. \textbf{64}, Amer. Math. Soc., Providence 2017.

\bibitem{BOAnnals} K. Bringmann and K. Ono,
\emph{Dyson's ranks and Maass forms}, Ann. of Math. {\bf 171} (2010), pages 419--449.

\bibitem{BOInventiones} K. Bringmann and K. Ono,
\emph{The $f(q)$ mock theta function conjecture and partition
ranks}, Invent. Math. \textbf{165} (2006), pages 243-266.

\bibitem{BOPNAS} K. Bringmann and K. Ono,
\emph{Lifting cusp forms to Maass forms with an application to partitions},
Proc. Natl. Acad. Sci., USA \textbf{104}, No. 10 (2007), pages 3725-3731.


\bibitem{BOR} K. Bringmann, K. Ono, and R. Rhoades,
\emph{Eulerian series as modular forms},
J. Amer. Math. Soc. {\bf 21} (2008), no. 4, pages 1085--1104.

\bibitem{Br} J. H. Bruinier, \emph{Borcherds products on
  $O(2,l)$ and Chern classes of Heegner divisors}, Springer Lecture
  Notes in Mathematics {\bf 1780}, Springer-Verlag (2002).

\bibitem{BF} J. H. Bruinier and J. Funke, \emph{On two geometric theta lifts},
Duke Math. J. {\bf 125} (2004), pages 45--90.

\bibitem{BruinierOno} J. H. Bruinier and K. Ono,
\emph{Heegner divisors, $L$-functions, and harmonic
weak Maass forms}, Ann. of Math. {\bf 172} (2010), no. 3, pages 2135--2181.

\bibitem{Deligne} P. Deligne. \emph{La conjecture de Weil}, I. Inst. Hautes \'{E}tudes Sci. Publ. Math., \textbf{47} (1974), pages 273--307.

\bibitem{Faber1} G. Faber, \emph{\"Uber polynomische entwickelungen},
Math. Ann. \textbf{57} (1903), pages 389--408.

\bibitem{Faber2} G. Faber, \emph{\"Uber polynomische entwickelungen, II},
Math. Ann. \textbf{64} (1907), pages 116--135.

\bibitem{Getz} J. Getz, \emph{A generalization of a theorem of Rankin and Swinnerton-Dyer on zeros of modular forms}, Proc. Amer. Math. Soc. \textbf{132} (2004), no. 8, pages 2221--2231 (electronic).


\bibitem{DLMF} NIST, \emph{Digital Library of Mathematical Functions}. \url{https://dlmf.nist.gov}.

\bibitem{OnoCBMS} K. Ono, \emph{The web of modularity: arithmetic of the coefficients of modular forms and $q$-series}, Conference Board of the Mathematical Sciences
\textbf{102}, Amer. Math. Soc. (2004).


\bibitem{OnoCDM} K. Ono, \emph{Unearthing the visions of a master: harmonic Maass forms and number theory}, Current Developments in Mathematics, International Press, Somerville, 2008, pages 347--454.

\bibitem{OnoMock}  K. Ono \emph{A mock theta function for the Delta-function}, Combinatorial Number Theory, Walter de Gruyter, Berlin, 2009, pages 141--155.

\bibitem{RSD} F. K. C. . Rankin and H. P. F. Swinnerton-Dyer, \emph{On the zeros of Eisenstein series}, Bull. London Math. Soc. {\bf 2} (1970), pages 169--170.

\bibitem{Rankin} R. A. Rankin, \emph{The zeros of certain Poincar\'e series}, Compositio Mathematica {\bf 46} (1982), pages 255--272.

\bibitem{Zwegers1} S. P. Zwegers, \emph{Mock theta functions},
Ph.D. Thesis, Universiteit Utrecht, 2002.




\end{thebibliography}
\end{document}